\documentclass[12pt]{amsart}
\usepackage{graphicx}
\usepackage{color}
\usepackage{latexsym}
\usepackage{amssymb}                %  AMS   -   Fonts and Symbols
\usepackage{amsmath}
\usepackage{amsthm}
\usepackage{amscd}
\usepackage{enumerate}
\usepackage{verbatim}
\usepackage{xy}
\xyoption{all}
\theoremstyle{plain}
\newtheorem{thm}{Theorem}
\newtheorem{lem}{Lemma}

\newtheorem{prop}{Proposition}
\newtheorem{cor}{Corollary} 
  
\theoremstyle{remark}

\numberwithin{equation}{section}
\def\ens{\ensuremath}                                     %ENSuremath
\newcommand\mtb[1]{\ens{\mathbb{#1}}}                     %\mtb=\ens{\mathbb{#1}} 

\newcommand{\N}{\mtb{N}}	   \newcommand{\Q}{\mtb{Q}}   	\newcommand{\R}{\mtb{R}}	
\newcommand{\T}{\mtb{T}}

\newcommand\mtc[1]{\ens{\mathcal{#1}}}           %%          \mtc=\ens{\mathcal{#1}}
                	
  \newcommand{\bet}{\ens{\beta}}
  
\newcommand{\gam}{\ens{\gamma}}	\newcommand{\lam}{\ens{\lambda}}
\newcommand\bld[1]{\ens{\boldsymbol{#1}}}

\newcommand{\card}{{\bf card}}                       %  \card
   
\newcommand\hl[1]{{\center \color{red} \ens{\bs\bs\bs} \\}} %%   \hl
\newcommand\hsp[1]{\mbox{}\hspace{#1mm}} %%%                     \hsp
\newcommand\vsp[1]{\par \vspace{#1mm}} %%%%%                     \vsp
\def\bs{{\bigstar}}                     %%                       BigStar    
\renewcommand{\v}{\vsp}

\newcommand{\h}{\hsp}

\newcommand{\rank}{\ens{\mathbf{rank}}}
\newcommand{\spann}{\ens{\mathbf{span}}}
\newcommand{\SAF}{\ens{\mathrm{SAF}}}
\newcommand{\iets}{\mtc G}

\newcommand{\dom}{\mathrm{domain}}
\newcommand{\dis}{\mathrm{disc}}

\newcommand{\ST}{\mtb K}

\title[Subgroup of interval exchanges]{Subgroup of interval exchanges \\ generated by
torsion elements \\ and rotations}
\author[M. Boshernitzan {}]{Michael Boshernitzan {}}
\address{Department of Mathematics, Rice University, Houston, TX~77005, USA}
\email{michael@rice.edu}
\thanks{The author was supported in part by research grant: NSF-DMS-1102298}
\date{Aug 3, 2012}
\begin{document}
\maketitle
%%%%%%%%%%%%%%%%%%%%%%%%%%%%%%%%%%%%%%%%%%%%%%%%%%%%%%%%%%%%%%%%%%%%%%%%%%%%%%%%%%%%%%%%%%
%%%%%%%%%%%%%%%%%%%%%%%%%%%%%%%%%
\begin{abstract}
Denote by $G$ the group of interval exchange transformations (IETs) on the unit 
interval. Let $G_{per}\subset G$ be the subgroup generated by torsion 
elements in $G$ (periodic IETs), and let $G_{rot}\subset G$ be the subset 
of $2$-IETs (rotations). 

The elements of the subgroup $G_1=\langle G_{per},G_{rot}\rangle\subset G$ 
(generated by the sets $G_{per}$ and $G_{rot}$) are characterized constructively 
in terms of their Sah-Arnoux-Fathi (SAF) invariant. The characterization implies 
that a non-rotation type $3$-IET  lies in $G_1$ if and only if the lengths of its 
exchanged intervals are linearly dependent over $\Q$. In particular, 
$G_1\subsetneq G$. 

The main tools used in the paper are the SAF invariant and a recent result by 
\mbox{Y.~Vorobets} that $G_{per}$ coincides with the commutator subgroup of $G$.
\end{abstract}
%%%%%%%%%%%%%%%%%%%%%%%%%%%%%%%%%%%%%%%%%%%%%%%%%%%%%%%%%%%%%%%%%%%%%%%%%%%%%%%%%%%%%%%%%%
%%%%%%%%%%%%%%%%%%%%%%%%%%%%%%%%%%%%%%%%%%%%%%%%%%%%%%%%%%%%%%%%%%%%%%%%%%%%%%%%%%%%%%%%%%
%%%%%%%%%%%%%%%%%%%%%%%%%%%%%%%%%%%%%%%%%%%%%%%%%%%%%%%%%%%%%%%%%%%%%%%%%%%%%%%%%%%%%%%%%%

%%%%%%%%%%%%%%%%%%%%%%%%%%%%%%%%%%%%%%%%%%%%%%%%%%%%%%%%%%%%%%%%%%%%%%%%%%%%%%%%%%%%%%%%%%
\section{A group of IETs}\label{sec:iets}
%%%%%%%%%%%%%%%%%%%%%%%%%%%%%%%%%%%%%%%%%%%%%%%%%%%%%%%%%%%%%%%%%%%%%%%%%%%%%%%%%%%%%%%%%%
%%%%%%%%%%%%%%%%%%%%%%%%%%%%%%%%%%%%%%%%%%%%%%%%%%%%%%%%%%%%%%%%%%%%%%%%%%%%%%%%%%%%%%%%%%
Denote by $\R$, $\Q$, $\N$ the sets of real, rational and natural numbers. 
 By a {\em standard interval} we mean a
finite interval of the form $X=[a,b)\subset\R$ (left closed - right open). 
We write $|X|=b-a$ for its length. 

By an IET (interval exchange transformation) we mean a pair $(X,f)$ where
$X=[a,b)$ is a standard interval and $f$ is a right continuous bijection 
$f\colon X\to X$ with a finite set $D$ of discontinuities and such that the 
translation function  $\gam(x)=f(x)-x$ is piecewise constant. 

The map $f$ itself is often referred to as IET, and then $X=\dom(f)$ and 
$D=\dis(f)$ denote the domain (also the range) of $f$ and the discontinuity 
set of $f$, respectively.

Given an IET  $f\colon X\to X$, the set $\dis(f)$ partitions $X$ into 
a finite number of subintervals  $X_k$ in such a way that $f$  
restricted to each $X_k$ is a translation
%%%%%%%%%%%%%<<<
\begin{equation}\label{eq:frest}
f|_{X_k}: x\to x+\gam_k,
\end{equation}
%%%%%%%%%%%%%>>>
so that the action of $f$ reduces to a rearrangement of the intervals $X_k$.
The number $r$ of exchanged intervals  $X_k$ can be specified by calling $f$ 
an $r$-IET. 

Denote by $\iets$ the set of IETs. Then
%%%%%%%%%%%%%<<<
\begin{equation}\label{eq:iets}
\iets=\bigcup_{r\geq1}\,\iets_r, 
\end{equation}
%%%%%%%%%%%%%>>>
where $\iets_r$ stands for the set of $r$-IETs:
%%%%%%%%%%%%%<<<
\begin{equation}\label{eq:riets}
      \iets_r=%\{f\in\iets\mid f \, \text{ is an }r\text{-IET}\}=
      \{f\in\iets\mid \card(\dis(f))=r-1\,\}.
\end{equation} 
%%%%%%%%%%%%%>>>

For a standard interval $X=[a,b)$, the subset 
%%%%%%%%%%%%%%%%%%<<<
\begin{equation}\label{eq:gx}
G(X)=\{f\in\iets\mid \dom(f)=X\}
\end{equation}
%%%%%%%%%%%%%%%%%%>>>
forms a group under composition (of bijections of $X$). Its identity is
$\bld1_{G(X)}\in\iets_1$, the identity map on $X$.
Given two standard intervals $X$ and $Y$, there is a canonical isomorphism
%%%%%%%%%%%%%%%%%%<<<
\begin{equation}\label{eq:phi}
\phi_{X,Y}: G(X)\to G(Y)
\end{equation}
%%%%%%%%%%%%%%%%%%>>>
defined by the formula
%%%%%%%%%%%%%%%%%%%%%%<<<
\begin{equation}\label{eq:is2}
\phi_{X,Y}(f)=\l\circ f\circ\l^{-1}\in G(Y),\quad\text{for }\,f\in G(X),
\end{equation}
%%%%%%%%%%%%%%%%%%%%%%%>>>
where $\l=\l_{X,Y}$ stands for the unique affine order preserving 
bijection $X\to Y$. 

By the group of IETs we mean the group\, $G:=G([0,1))$.
(It is isomorphic to $G(X)$, for any standard interval $X$).

The interval exchange transformations have been a popular subject of study in 
ergodic theory. (We refer the reader to the book \cite{Via} by Marcelo Viana 
which may serve a nice introduction and survey reference in the subject). Most 
papers on IETs study these as dynamical systems; they concern specific 
dynamical properties (like minimality, ergodicity, mixing properties etc.) 
the IETs may satisfy.

%%%%%%%%%%%%%%%%%%%%%%%%%%%%%%%%%%%%%%%%%%%%%%%%%%%%%%%%%%%%%%%%%%%%%%%%%%%%%%%%%%%%%%%%%%
The focus of the present paper is different; we address certain questions on the
group-theoretical structure of the group $G$ of IETs. 
(For recent results on this general area see \cite{Vor} and \cite{Nov}).
In particular, we discuss possible generator subsets of the group $G$.

It was known for a while that the subgroup $G_{per}$ generated by periodic IETs forms 
a proper subgroup of $G$; in particular, $G_{per}$ contains no irrational rotations. 
(The SAF invariant introduced in the next section vanishes on $G_{per}$ but has non-zero 
value on irrational rotations).
On the other hand, the set $G_{per}$ contains some uniquely ergodic,
even pseudo-Anosov (self-similar) IETs (see \cite{AY}).

We show that the subgroup $G_1=\langle G_{per}, G_{rot}\rangle$ generated by 
$G_{per}$ and the set of rotations $G_{rot}\subset G$ 
is still a proper subgroup of $G$. On the other hand, $G_1$ is large enough
to contain all  rank $2$ IETs (see Section \ref{sec:rank2}), 
in particular the IETs over quadratic number fields.

We present a constructive criterion (in terms of the SAF invariant)
for a given IET to lie in~$G_{1}$. It follows from this criterion that
a $3$-IET lies in $G_1$ if and only if the lengths of its exchanged subintervals
are linearly independent.
Note that $3$-IETs generate the whole group of IETs. (More precisely, $\iets_3\cap G$ is 
a generating set for the group $G$).
%%%%%%%%%%%%%%%%%%%%%%%%%%%%%%%%%%%%%%%%%%%%%%%%%%%%%%%%%%%%%%%%%%%%%%%%%%%%%%%%%%%%%%%%%%
%%%%%%%%%%%%%%%%%%%%%%%%%%%%%%%%%%%%%%%%%%%%%%%%%%%%%%%%%%%%%%%%%%%%%%%%%%%%%%%%%%%%%%%%%%
\section{The $\SAF$ invariant and subgroups of $G(X)$}
%%%%%%%%%%%%%%%%%%%%%%%%%%%%%%%%%%%%%%%%%%%%%%%%%%%%%%%%%%%%%%%%%%%%%%%%%%%%%%%%%%%%%%%%%%
Throughout the paper (whenever vector spaces or linear dependence/independence 
are discussed) the implied field (if not specified) is always meant to be $\Q$, 
the field of rationals.

Denote by $\T$ the tensor product of two copies of reals viewed
as vector spaces (over $\Q$). 
Denote by $\ST$ the skew symmetric tensor product of two copies of reals:
%%%%%%%%%%%%%%%%%%%%%%%<<<
\begin{equation*}%\label{eq:T}
\T:=\R\otimes_{\Q}\R; \qquad  \ST:=\R\wedge_{\Q}\R\subset\T.
\end{equation*}
Recall that $\ST$ is the vector subspace of $\T$ spanned by 
the wedge products
%%%%%%%%%%%%%%%%%%%%%%%<<<
\begin{equation*}%\label{eq:wed}
u\wedge v:= u\otimes v-v\otimes u, \qquad u,v\in\R.
\end{equation*}

%%%%%%%%%%%%%%%%%%%%%%%>>>
The {\em Sah-Arnoux-Fathi (SAF) invariant} (sometimes also called the 
{\em scissors congruence invariant}) of 
$f\in\iets_{r}$ is defined by the formula                  %\mbc{eq:saff} 
%%%%%%%%%%%%%%%%%%%%%%%%%%<<<
\begin{equation}\label{eq:saff}
\SAF(f):=\sum_{k=1}^r\, \lam_k\otimes\gam_{k}\in \T, 
\end{equation}
%%%%%%%%%%%%%%%%%%%%%%%%%%>>>
where the vectors $\vec \lam=(\lam_1,\lam_2,\ldots,\lam_r,)$, 
$\vec \gam=(\gam_1,\gam_2,\ldots,\gam_r,)\in\R^r$ encode the lengths 
\mbox{$\lam_k=|X_k|$} of exchanged intervals $X_k$ and the corresponding 
translation constants $\gam_{k}$, respectively (see \eqref{eq:frest}).

The SAF invariant was introduced independently by Sah \cite{Sah} and Arnoux and 
Fathi \cite{Arn}. The following lemma makes this invariant a useful tool in 
the study of IETs.

%%%%%%%%%%%%%%%%%%%%%%%%%%<<<
\begin{lem}\label{lem:hom}
Let $X$ be a standard interval. Then 
\begin{itemize}
\item[(a)] $\SAF\colon G(X)\to \T$ is a group homomorphism;\v2
\item[(b)] $\SAF(G(X))=\ST\subset\T$;\v3
\item[(c)] $\SAF\colon G(X)\to \ST$ is a surjective group homomorphism. 
\item[]\h3(This summarizes (a) and (b)).
\end{itemize}¥
\end{lem}
%%%%%%%%%%%%%%%%%%%%%%%%%%>>>

%%%%%%%%%%%%%%%%%%%%%%%%%%%%%%%%%%%%%%%%%%%%%%%%%%%%%%%%%%%%%%%%%%%%%%%%%%%%%%%%%%%%%%%%%%
\subsection{Subgroups of $G(X)$}
Let $X$ be a standard interval. An IET $f\in G(X)$ is called periodic if $f$ is 
an element of finite order in the group $G(X)$  (defined in \eqref{eq:gx}). Set
%%%%%%<
\begin{equation}\label{eq:rrr}
\iets_r(X):=\iets_r\cap G(X), \quad r\geq1,
\end{equation}
%%%%%%>
(see notation \eqref{eq:iets}). 
Note that $\iets_1(X)=\{\bold1_{G(X)}\}$ is a singleton.

\pagebreak[3]
%%%%%%%%%%%%%%%%%%%%%%%%%%%%%%%%%%%%%%%%%%%%%%%%%%%%%%%%%%%%%%%%%%%%%%%%%%%%%%%%%%%%%%%%%%
Consider the following subgroups of $G(X)$:
%%%%%%%%<<<
\begin{itemize}
\item $G'(X)=[G(X),G(X)]$ is the commutator subgroup of $G$ (generated by 
	  the \\ \hsp3 commutators $f^{-1}g^{-1}fg$, with $f,g\in G(X))$;\v3
\item $G_{per}(X)$ is the subgroup of $G(x)$ generated by 
      periodic IETs $f\in G(X)$;\v5
\item $G_0(X)=\{f\in G(X)\mid \SAF(f)=0\}$ (see \eqref{eq:saff});\v3
\item $G_{rot}(X)=\iets_1(X)\cup\iets_2(X)=\iets_2(X)\cup\{\bld1_{G(X)}\}$
      is the group of rotations on $X$.\\
      (It contains  all IETs  $f\in G(X)$ with at most one discontinuity).
\end{itemize}
%%%%%%%%>>>

Observe that (a) in Lemma \ref{lem:hom} implies immediately the inclusions 
\[
G'(X)\subset G_0(X); \quad  G_{per}(X)\subset G_0(X),
\]
as well as 
the fact that the set $G_0(X)$ forms a subgroup of $G(X)$.
%%%%%%%%%%%%%%%%%%%%%%%%%%%%%%%%%%%%%%%%%%%%%%%%%%%%%%%%%%%%%%%%%%%%%%%%%%%%%%%%%%%%%%%%%%

An unpublished theorem of Sah \cite{Sah} (mentioned by Veech in \cite{Ve_SAF}) 
contains Lemma~\ref{lem:hom} and  the equality $G'(X)=G_0(X)$. 
This equality has been recently extended by Vorobets \cite{Vor}
to also include $G_{per}(X)$ as an additional set:
%%%%%%%%%%%%%%%%<<<
\begin{equation}\label{eq:3g}
G'(X)=G_0(X)=G_{per}(X).
\end{equation}
%%%%%%%%%%%%%%%%>>>
We refer to Vorobets's paper \cite{Vor} for a nice self-contained 
introduction to the SAF invariant. In particular, the paper contains 
the proof of Lemma \ref{lem:hom} and of the equality \eqref{eq:3g}.

%%%%%%%%%%%%%%%%%%%%%%%%%%%%%%%%%%%%%%%%%%%%%%%%%%%%%%%%%%%%%%%%%%%%%%%%%%%%%%%%%%%%%%%%%%
The results of the present paper concern the subgroup
%%%%%%%%%%%%%%%%%%<<<
\begin{equation}\label{eq:g1x}
G_1(X)=\langle G_{per}(X),G_{rot}(X)\rangle\subset G(X) 
\end{equation}
%%%%%%%%%%%%%%%%%%>>>
generated by periodic IETs and the rotations in $G(X)$. The elements of $G_1(X)$ are 
classified in terms of their SAF invariant, see Theorem \ref{thm:g1saf}. 
It is shown that ``most'' $3$-IETs do not lie in $G_1(X)$ (Lemma \ref{lem:ex3} and
Theorem \ref{thm:crit}). In particular, it follows that $G_1(X)\neq G(X)$.
%%%%%%%%%%%%%%%%%%%%%%%%%%%%%%%%%%%%%%%%%%%%%%%%%%%%%%%%%%%%%%%%%%%%%%%%%%%%%%%%%%%%%%%%%%
\section{Subgroup $G_1$ and its $\SAF$ invariant characterization.}
%%%%%%%%%%%%%%%%%%%%%%%%%%%%%%%%%%%%%%%%%%%%%%%%%%%%%%%%%%%%%%%%%%%%%%%%%%%%%%%%%%%%%%%%%%
For $u\in\R$, denote 
\[
\ST(u)=\{u\wedge v\mid v\in\R\}\subset\ST=\R\wedge_{\Q}\R.
\]
$\ST(u)$ forms a vector subspace of $\ST$. If $u\neq0$, $\ST(u)$ is isomorphic to  $\R/\Q$.

In the next lemma we compute the SAF invariant for $2$-IETs. This  is known
and follows immediately from the definition \eqref{eq:saff}, but is included for 
completeness.

Recall that $\iets_r(X)=\iets_r\cap G(X)$ stands for the set of $r$-IETs on $X$.
%%%%%%%%%%%%<<<
\begin{lem}\label{lem:safrot} Let $X$ be a standard interval. Assume that 
$f\in \iets_2(X)$  exchanges two subintervals of lengths $\lam_1$ and $\lam_2$  
(with $\lam_1+\lam_2=|X|$). Then    
\[
\SAF(f)=|X|\wedge\lam_{1}\in \ST(|X|).
\]
\end{lem}
%%%%%%%%%%%%>>>
%%%%%%%<<<
\begin{proof}
For such $f$ the traslation constants are $\gam_1=-\lam_2$ and 
$\gam_2=-\lam_2+1=\lam_1$  (see \eqref{eq:frest}). It follows 
that
%%%<
\begin{align*}
\SAF(f)=\lam_1\otimes\gam_1+\lam_2\otimes\gam_2&=
      \lam_1\otimes(-\lam_2)+\lam_2\otimes\lam_1=\\
      &=\lam_2\wedge\lam_1=(\lam_2+\lam_1)\wedge\lam_1=|X|\wedge\lam_{1}.
\end{align*}
%%%>
\end{proof} 
%%%%%%%>>>
%%%%%%%%%%%%<<<
\begin{lem}\label{lem:safrot2}
Let $X$ be a standard interval. Let $\bet\in\ST(|X|)$. Then there 
exists $f\in \iets_2(X)$  such that $\SAF(F)=\bet$.
\end{lem}
%%%%%%%%%%%%>>>
%%%%%%%%%%%%%%%%%%%%%%%%%%%%%%%%%%%%%%%%%%%%%%%%%%%%%%%%%%%%%%%%%%%%%%%%%%%%%%%%%%%%%%%%%%
%%%%%%%<<<
\begin{proof}
Let $\bet=|X|\wedge t$. Select a rational $r\in\Q$  such that $0<r|X|+t<|X|$.
Set $\lam_1=r|X|+t$ and $\lam_2=|X|-\lam_1$. Take $f\in\iets_2(X)$ which exchanges
two intervals of lengths $\lam_1$  and $\lam_2$. Then, by Lemma \ref{lem:safrot},
\[
\SAF(f)=|X|\wedge\lam_1=|X|\wedge(r|X|+t)=|X|\wedge t,
\]
completing the proof.
\end{proof}
%%%%%%%>>>
%%%%%%%%%%%%<<<
\begin{cor}\label{cor:rot}
Let $X$ be a standard interval. Then
\[
\SAF(G_{rot}(X))=\SAF(\iets_{2}(X))=\ST(|X|).
\]
\end{cor}
%%%%%%%%%%%%>>>
\begin{proof}
Follows from Lemmas \ref{lem:safrot} and  \ref{lem:safrot2}.
\end{proof}
%%%%%%%%%%%%%%%%%%%<<<
\begin{thm}\label{thm:gh}
Let $X$ be a standard interval and let $f\in G(X)$. Assume that 
$\SAF(f)\in \ST(|X|)$. Then there exists a rotation $g\in\iets_{2}(X)$ and
$h_1,h_2\in G_{per}(X)$  such that
\[
f=g\circ h_2=h_1\circ g.
\]
\end{thm}
%%%%%%%%%%%%%%%%%%%>>>
%%%%%%%%%%%%%%%%%%%%%%%%%%%%%%%%%%%%%%%%%%%%%%%%%%%%%%%%%%%%%%%%%%%%%%%%%%%%%%%%%%%%%%
Note that in the above theorem $G_{per}(X)$ can be replaced by $G'(X)$ or  $G_0(X)$
(see~\eqref{eq:3g}).
%%%%%%%%%%%%%%%%%%%<<<
\begin{proof}[Proof of Theorem \em\ref{thm:gh}]

By Lemma \ref{lem:safrot2} there exists $g\in\iets_2(X)$ such that 
$\SAF(g)=\SAF(f)\in \ST(|X|)$. Take
%%%%%%%%%%%%<<<
$
h_1=f\circ g^{-1},\ h_2=g^{-1}\circ f.
$
%%%%%%%%%%%%>>>
Then $h_1,h_2\in G_0(X)=G_{per}(X)$ in view of Lemma~\ref{lem:hom}.
\end{proof}
%%%%%%%%%%%%%%%%%%%>>>

%%%%%%%%%%%%%%%%<<<
\begin{thm}\label{thm:g1saf}
Let $X$ be a standard interval and let $f\in G(X)$. Then 
\[
f\in G_1(X)\ \Longleftrightarrow\  \SAF(f)\in \ST(|X|).
\]
\end{thm}
%%%%%%%%%%%%%%%%>>>
%%%%%%%%%%%%%%%%%%%%%%%%%%%%%%%%%%%%%%%%%%%%%%%%%%%%%%%%%%%%%%%%%%%%%%%%%%%%%%%%%%%%%%%%%%
%%%%%%%%%<<<
\begin{proof} Direction $\Rightarrow$. In view of \eqref{eq:g1x}, it is enough
to show the following two inclusions: 
\[
S_1=\SAF(G_{rot}(X))\in \ST(|X|); \quad S_2=\SAF(G_0(X))\in \ST(|X|).
\]
Both are immediate: $S_{1}=\ST(|X|)$ by Corollary \ref{cor:rot}, and $S_2=0$.

Direction $\Leftarrow$. Follows from Theorem \ref{thm:gh}.
\end{proof}
%%%%%%%%%>>>

%%%%%%%%%%%%%%%%%%%%%%%%%%%%%%%%%%%%%%%%%%%%%%%%%%%%%%%%%%%%%%%%%%%%%%%%%%%%%%%%%%%%%%%%%%
%%%%%%%%%%%%%%%%%%%%%%%%%%%%%%%%%%%%%%%%%%%%%%%%%%%%%%%%%%%%%%%%%%%%%%%%%%%%%%%%%%%%%%%%%%
\section{Rank 2 IETs lie in $G_{1}$}\label{sec:rank2}
By the span of a subset  $A\subset\R$ (notation: $\spann(A)$) we mean
the minimal vector subspace of $\R$ (over $\Q$) containing $A$.

By the rank of an IET $f$ (notation $\rank(f))$ we mean the dimension of 
the span of the set the lengths of the intervals exchanged by $f$:
\[
       \rank(f):=\dim_{Q}(L(f)),
\]
where
\begin{equation}\label{eq:lf}
       L(f):=\spann(\{\lam_1,\lam_2,\ldots,\lam_r\}) \quad \text{(if $f\in\iets_r$).}
\end{equation}

The following implication is immediate: \
 $\rank(f)=1\ \implies\  f\in G_{per}=G_{0}$.
%%%%%%%%%%%%%<<<
\begin{thm}\label{thm:r2}
For $f\in G(X)$, the following implication holds:  
\[
\rank(f)\leq2 \quad\implies\quad f\in G_{1}(X).
\]
\end{thm}
%%%%%%%%%%%%%>>>
%%%%%%%%%%%%%%%%%%%%%%%%%%%%%%%%%%%%%%%%%%%%%%%%%%%%%%%%%%%%%%%%%%%%%%%%%%%%%%%%%%%%%%%%%%
%%%%%%%%%%%%%%%%%%%%<<<
\begin{proof}
We may assume that $\rank(f)=\dim(L(f))=2$. 
Since all $\lam_k\in L(f)$, it follows that $|X|=\sum_{k=1}^{r}\lam_{k}\in L(F)$.
Let $B=\{|X|, u\}$ be a basis in $L(F)$. Then 
\begin{align*}
\SAF(f)\in L(f)\wedge_{\Q}L(f)&=\{q\,|X|\wedge u\mid q\in\Q\}=\\
    &=\{|X|\wedge qu\mid q\in\Q\}\subset\ST(|X|),
\end{align*}
whence $f\in G_1(X)$, in view of  Theorem \ref{thm:g1saf}.
\end{proof}
%%%%%%%%%%%%%%%%%%%%>>>
%%%%%%%%%%%%%%%%%%%%%%%%%%%%%%%%%%%%%%%%%%%%%%%%%%%%%%%%%%%%%%%%%%%%%%%%%%%%%%%%%%%%%%%%%%
%%%%%%%%%%%%%%%%%%%%%%%%%%%%%%%%%%%%%%%%%%%%%%%%%%%%%%%%%%%%%%%%%%%%%%%%%%%%%%%%%%%%%%%%%%
\section{Criterion for $3$ IETs to lie in $G_1$}
Let $f\in\iets_3(X)$. Then $f$ has $2$ discontinuities, and $f$ acts by reversing the 
order of three subintervals, $X_1$, $X_2$ and $X_3$. Set $\lam_i=|X_i|$. Then 
$X=\dom(f)=[a,a+\lam_1+\lam_1+\lam_3)$, with some $a\in\R$.
Set for convenience $a=0$ so that $X=[0,\lam_1+\lam_1+\lam_3)$.

The corresponding translation constants are easily computed:
%%%%%%%%%%%%<<<
\begin{align*}
\gam_{1}&=f(0)-0=\lam_2+\lam_3=|X|-\lam_1; \\ 
\gam_2&=f(\lam_{1})-\lam_{1}=\lam_3-\lam_1; \\
\gam_{3}&=f(\lam_1+\lam_2)-(\lam_1+\lam_2)=0-(\lam_1+\lam_2)=\lam_3-|X|.
\end{align*}
%%%%%%%%%%%%<<<
(The above short computation follows general receipt for computing the translation 
constants $\gam_{k}$ in terms of the constants $\lam_{k}$ and the permutation 
specifying the way intervals $X_{k}$ are rearranged, 
see e.g. \cite[Section 0]{Ve_IET}. The permutation here must be $(321)$).

Since $\lam_2=|X|-\lam_1-\lam_3$, we can compute $\SAF(X)$ in terms of linear
combinations of the wedge products involving only $X$, $\lam_1$ and $\lam_3$:
%%%%%%%%<<<
\begin{align}\label{eq:safkx}
\SAF(X)=\sum_{k=1}^3\ \lam_k\otimes\gam_k=
|X|\wedge(\lam_1-\lam_3)-\lam_1\wedge\lam_3.
\end{align}
%%%%%%%%>>>
We need the following known basic fact  (see e.g. \cite[Lemma 3.1]{Vor}).
%%%%%%%%%%%%%%%%%%%<<<
\begin{lem}\label{lem:vorob}
If $v_1,v_2,\ldots,v_n$ are $n\geq2$ linearly independent real numbers, 
then the $\frac{n(n-1)}2$ wedge products  $v_{i}\wedge v_{j}$, $1\leq i<j\leq n$,
are linearly independent.
\end{lem}
%%%%%%%%%%%%%%%%%%%>>>
%%%%%%%%%%%%%%%%%%%%%%%%%%%%%%%%%%%%%%%%%%%%%%%%%%%%%%%%%%%%%%%%%%%%%%%%%%%%%%%%%%%%%%%%%%
%%%%%%%%%%%%%%%%%%%<<<
\begin{cor}\label{cor:basic}
Let $v_1,v_2,v_3\in\R$ be linearly independent. Then  $v_1\wedge v_2\neq v_3\wedge u$
for all $u\in\R$.
\end{cor}
%%%%%%%%%%%%%%%%%%%>>>
\begin{proof}
Assume to the contrary that  
\[
v_1\wedge v_2= v_3\wedge u
\]
for some $u\in\R$. If $u\notin\spann(\{v_1,v_2,v_3\})$, then the set 
$\{u,v_1,v_2,v_3\}$ is linearly independent, contradicting Corollary \ref{cor:basic}.

And if  $u\in\spann(\{v_1,v_2,v_3\})$, 
then $u=q_1v_1+q_2v_2+q_3v_3$ with some $q_i\in\Q$ whence
\[
v_1\wedge v_2=q_1\,v_3\wedge v_1+q_2\,v_3\wedge v_2,
\]
a contradiction with Lemma \ref{lem:vorob} again.
\end{proof}
%%%%%%%%%%%%%%%<<<
\begin{lem}\label{lem:ex3}
Let $f\colon X\to X$ be a $3$-IET  and assume that $\rank(f)=3$. (Equivalently,
the set $\{\lam_1,\lam_1,\lam_3\}$ is linearly independent).
Then $f\notin G_1(X)$.
\end{lem}
%%%%%%%%%%%%%%%>>>
\begin{proof}
Assume to the contrary that $f\in G_1(X)$. Then, by Theorem \ref{thm:g1saf},
$\SAF(f)\in\ST(|X|)$. It follows from \eqref{eq:safkx} that  
$\lam_1\wedge\lam_3\in\ST(|X|)$, i.e. that 
%%%%%%%%%<<<
\begin{equation}\label{eq:wed}
\lam_1\wedge\lam_3=|X|\wedge u,
\end{equation}
%%%%%%%%%>>>
for some $u\in\R$. Since the numbers $\lam_1,\lam_3$ and $|X|=\lam_1+\lam_2+\lam_3$ 
are linearly independent, \eqref{eq:wed} contradicts Corollary \ref{cor:basic}.

\end{proof}
\begin{thm}[Criterion for a $3$-IET to lie in $G_1$]\label{thm:crit}
Let $f\colon X\to X$ be a $3$-IET. Then 
\[
f\in G_1(X)\quad\iff\quad\rank(f)\leq2.
\]
\end{thm}
\begin{proof}
The direction $\Rightarrow$ is a contrapositive restatement of Lemma \ref{lem:ex3}.
The direction $\Leftarrow$\, is given by Theorem \ref{thm:r2}.
\end{proof}
\section{Validation of the membership in classes $G_0(X)$ and $G_1(X)$}
%%%%%%%%%%%%%%%%%%%%%%%%%%%%%%%%%%%%%%%%%%%%%%%%%%%%%%%%%%%%%%%%%%%%%%%%%%%%%%%%%%%%%%%%%%
Given $f\in G(X)$, we describe constructive procedures to decide whether 
$f\in G_{per}(X)$ and whether $f\in G_1(X)$.
%%%%%%%%%%%%%%%%%%%
\subsection{Does the inclusion $f\in G_{per}(X)$ hold?} Since $G_{per}(X)=G_0(X)$ 
(see \eqref{eq:3g}), one only has to test the equality 
\[
\SAF(f):=\sum_{k=1}^r \lam_k\otimes\gam_{k}=0.
\]
We assume that the {\em linear structure} of $f$ is known.
By the linear structure of $f$ we mean:\v1

\hsp5  (a) a basis  $B=\{v_1,v_2,\ldots,v_n\}$ of the finite dimensional space 
\[
       L(f):=\spann(\{\lam_1,\lam_2,\ldots,\lam_r\});
\]

\hsp5  (b) the (unique) linear representations $\lam_{k}=\sum_{i=1}^n q_{k,i}v_i$, 
       $1\leq k\leq r$, with all $q_{k,i}\in\Q$. 
  
Observe that all translation constants $\gam_k$ also lie in $L(f)$, and their linear 
representation in terms of basis $B$ can be computed. This way one can get presentation
%%%%%%%%%%%%%%<<
\begin{equation}\label{eq:crit}
\SAF(f)=\sum_{1\leq i<j\leq n}p_{i,j}\ v_i\wedge v_j,
\end{equation}
%%%%%%%%%%%%%%>>
with known $p_{i,j}\in\Q$. By Lemma \ref{lem:vorob}, $\SAF(f)=0$ if and only if 
all the constants $p_{i,j}$ vanish.
\[
f\in G_{per}(X)=G_0(X)\quad
\iff \quad p_{i,j}=0, \text{ for all} \ i,j.
\]
	
%%%%%%%%%%%%%%%%%%%%%%%%%%%%%%%%%%%%%%%%%%%%%%%%%%%%%%%%%%%%%%%%%%%%%%%%%%%%%%%%%%%%%%%%%%
\subsection{Does the inclusion $f\in G_1(X)$ hold?}

Again we assume that the linear structure of $f$ is known. 
To answer the question, we proceed as follows.

First, we modify the basis $B$ of $L(f)$  to make $v_1=|X|$. 
Then we proceed just as before and get \eqref{eq:crit} with
known $p_{i,j}\in\Q$. By Theorem \ref{thm:g1saf} and Lemma \ref{lem:vorob}, 
%%%%%%%%%%<<<
\[
f\in G_1(X)\quad\iff\quad\SAF(f)\in\ST(|X|)\quad
\iff \quad p_{i,j}=0, \text{ for } 2\leq i<j\leq n.
\]
%%%%%%%%%%>>>
%%%%%%%%%%%%%%%%%%%%%%%%%%%%%%%%%%%%%%%%%%%%%%%%%%%%%%%%%%%%%%%%%%%%%%%%%%%%%%%%%%%%%%%%%%
%%%%%%%%%%%%%%%%%%%%%%%%%%%%%%%%%%%%%%%%%%%%%%%%%%%%%%%%%%%%%%%%%%%%%%%%%%%%%%%%%%%%%%%%%%
\section{Class $G_1$ is not preserved under induction} 
An IET  $(X,f)$  is called minimal if its every orbit is dense in $X$.
For a sufficient condition for an IET to be minimal see \cite{Kea_IET}.

Let $(X,f)$  be an $r$-IET and assume that $Y\subset X$ is 
a standard subinterval. It is well known (see e.g. \cite{CFS} or \cite{Kea_IET}) that 
the (first return) map  $f_{Y}\colon Y\to Y$ induced by $f$ on $Y$ is also an IETs
(exchanging at most $r+1$ subintervals). It is known that, under the minimality
assumption, the $\SAF$ invariant is preserved under induction.
%%%%%%%%%%%%%%%%%%<<<
\begin{prop}\label{prop:indsaf} 
{\em(\cite{Arn}, Part II, Proposition 2.13)}
Let $(X,f)$ be a minimal IET and assume that $Y\subset X$ is 
a standard subinterval. Then 
%%%%%%%%%%<<<
\[
\SAF(f)=\SAF(f_Y).
\]
%%%%%%%%%%>>>
\end{prop}
%%%%%%%%%%%%%%%%%%>>>
%%%%%%%%%%%%%%%%%%%%%%%%%%%%%%%%%%
%%%%%%%%%%%%%%%%%%<<<
\begin{cor} 
Let $(X,f)$ be a minimal IET and assume that $Y\subset X$ is 
a standard subinterval. Then 
%%%%%%%%%%<<<
\[
f\in G_{per}(X)\quad\iff\quad  f_Y\in G_{per}(X)
\]
%%%%%%%%%%>>>
\end{cor}
%%%%%%%%%%%%%%%%%%>>>
%%%%%%%%%%<<<
\begin{proof}
Follows from Proposition \ref{prop:indsaf} and the fact that $G_{per}(X)=G_0(X)$.
\end{proof}
%%%%%%%%%%>>>

%%%%%%%%%%%%%%%%%%<<<
\begin{cor} 
Let $(X,f)$ be a minimal IET and assume that $Y\subset X$ is 
a standard subinterval such that  $\frac{|Y|}{|X|}\in\Q$.
Then 
%%%%%%%%%%<<<
\[
f\in G_1(X)\quad\iff\quad  f_Y\in G_1(X).
\]
%%%%%%%%%%>>>
\end{cor}
%%%%%%%%%%%%%%%%%%>>>
%%%%%%%%%%<<<
\begin{proof}
Follows from Proposition \ref{prop:indsaf} because $\ST(|X|)=\ST(|Y|)$ assuming that
$\frac{|Y|}{|X|}\in\Q$.
\end{proof}
%%%%%%%%%%>>>
%%%%%%%%%%%%%%%%%%<<<
\begin{thm}
Let $(X,f)$ be a minimal IET. Then the following two assertions are equivalent:

(a) There exists a standard subinterval $Y\subset X$ such that $f_{Y}\in G_1(Y)$;

(b) $\SAF(X)=u\wedge v$, for some $u,v\in \R$. 
\end{thm}
%%%%%%%%%%%%%%%%%%>>>
One can show that if (b) of the above theorem holds and $\SAF(X)\neq0$ then $u,v\in L(f)$.
\begin{proof}
(a)$\Rightarrow$(b). $\SAF(X)=\SAF(Y)\in\ST([Y])$  whence  $\SAF(X)=|Y|\wedge v$, for some $v\in\R$.

(b)$\Rightarrow$(a). Let $\SAF(X)=u\wedge v$, for some $u,v\in\R$. Without loss of 
   generality,both $u,v$ can be selected positive (using the identities 
   $u\wedge v=(-v)\wedge u$ and  $1\wedge1=0$). Select $q\in \Q$ so that $0<qu<|X|$. 
   Select any standard subinterval $Y\subset X$ of length $|Y|=qu$. Then 
   $\SAF(f_Y)=u\wedge v=|Y|\wedge (q^{-1}v)\in\SAF(|Y|)$, and hence $f_{Y}\in G_1(Y)$.
\end{proof}

Let $K\subset \R$ be a subfield of reals. An IET $f$ is said to be over $K$ if  
$L(f)\subset K$ (see \eqref{eq:lf}), i.e. if all (lengths of exchanged intervals)
$\lam_k$ lie in $K$.

We complete the paper by the following result. 
%%%%%%%%%%%%%%%%%%%%%%%%%%%%%%%%%%%%%%%%%%%%%%%%%%%%%%%%%%%%%%%%%%%%%%%%%%%%%%%%%%%%%%%%%%
%%%%%%%%%%%%%%%%%%<<<
\begin{thm}\label{thm:fin}
Let $K$ be a real quadratic number field and let  
$(X,f)$ be a minimal IET over $K$.
Let $Y\subset X$ be a standard subinterval. Then
\[
f_Y\in G_1(Y)\quad\iff\quad |Y|\in K.
\]
\end{thm}
%%%%%%%%%%%%%%%%%%>>>
%%%%%%%%<<<
\begin{proof}
Since $\rank(f)=\dim(L(f))\leq\dim(F)=2$, and $\rank(f)\neq1$ (because $f$ is not 
periodic), we conclude that $\rank(f)=2$ and $L(f)=K$. 

{\bf Proof of the $\Leftarrow$ \ implication}.
Select a basis $B=\{|Y|,u\}$ in $K$. Then
\[
\SAF(f_{Y})=\SAF(f)\in K\wedge K=\{q\,|Y|\wedge u\mid q\in\Q\}\subset \ST(|Y|).
\]
By Theorem \ref{thm:g1saf}, $f_{Y}\in G_1(Y)$.\v1

{\bf Proof of the $\Rightarrow$ implication}. 
It has been proved in \cite{Bo_rank2} that minimal rank $2$ IETs must be 
uniquely ergodic. Thus $f$ is uniquely ergodic.
By McMullen theorem \cite[Theorem 2.1]{McM}, $\SAF(f)\neq0$. (McMullen uses the 
``Galois flux'' invariant for the IETs over a quadratic number field which, 
in his setting, is equivalent to the SAF invariant). 

Select a basis $B=\{|X|,u\}$ in $K$.
Then 
\[
\SAF(f_{Y})=\SAF(f)\in K\wedge K=\{q\,|X|\wedge u\mid q\in\Q\}.
\]
Since $\SAF(f)\neq0$, $\SAF(f_{Y})=q\,|X|\wedge u$, with some $q\in\Q$, $q\neq0$.

By Theorem \ref{thm:g1saf}, $f_Y\in G_1(Y)$ implies that
$q\,|X|\wedge u=|Y|\wedge v$,
for some $v\in\R$. This is incompatible with the assumption $Y\notin K$ in view 
of Corollary  \ref{cor:basic}, completing the proof.
\end{proof}
%%%%%%%%>>>

\v9

%%%%%%%%%%%%%%%%%%%%%%%%%%%%%%%%%%%%%%%%%%%%%%%%%%%%%%%%%%%%%%%%%%%%%%%%%%%%%%%%%%%%%%%%%%
%%%%%%%%%%%%%%%%%%%%%%%%%%%%%%%%%%%%%%%%%%%%%%%%%%%%%%%%%%%%%%%%%%%%%%%%%%%%%%%%%%%%%%%%%%

\end{document}